\documentclass[a4paper,12pt, reqno]{amsart}
\usepackage{a4wide}
\usepackage{amsthm}
\usepackage{amssymb}
\usepackage{palatino}
\usepackage{lipsum}
\usepackage{longtable}
\usepackage{enumitem}
\usepackage{hyperref}
\usepackage{rotating}
\newtheorem{thm}{Theorem}[section]

\newtheorem{lem}[thm]{Lemma}

\theoremstyle{definition}
\newtheorem{defn}[thm]{Definition}

\newtheorem{exmp}[thm]{Example}

\theoremstyle{remark}

\numberwithin{equation}{section}

\newcommand{\legendre}[2]{\genfrac{(}{)}{}{}{#1}{#2}}
\newcommand{\Mod}[1]{\ (\mathrm{mod}\ #1)}

\begin{document}
\title[Lucas non-Wieferich primes in arithmetic progressions and the $abc$ conjecture]{Lucas non-Wieferich primes in arithmetic progressions and the $abc$ conjecture}
\author[K. Anitha, I. Mumtaj Fathima and A R Vijayalakshmi]{K. Anitha$^{(1)}$, I. Mumtaj Fathima$^{(2)}$ and A R Vijayalakshmi$^{(3)}$}
\address{$^{(1)}$Department of Mathematics, SRM IST Ramapuram, Chennai 600089, India}
\address{$^{(2)}$Research Scholar, Department of Mathematics, Sri Venkateswara College of Engineering \\ Affiliated to Anna University, Sriperumbudur, Chennai 602117, India}
\address{$^{(3)}$Department of Mathematics, Sri Venkateswara College of Engineering, Sriperumbudur, Chennai 602117, India}
\email{$^{(1)}$subramanianitha@yahoo.com}
\email{$^{(2)}$tbm.fathima@gmail.com}
\email{$^{(3)}$avijaya@svce.ac.in}


\begin{abstract}
We prove the lower bound for the number of Lucas non-Wieferich primes in arithmetic progressions. More precisely, for any given integer $k\geq 2$ there are $\gg \log x$ Lucas non-Wieferich primes $p\leq x$ such that $p\equiv\pm1\Mod{k}$, assuming the $abc$ conjecture for number fields. Further, we discuss some applications of Lucas sequences in Cryptography.
\end{abstract}

\maketitle
\section{Introduction}

Let $a\geq2$ be an integer. An odd rational prime $p$ is said to be a \textit{Wieferich prime for base $a$} if
\begin{equation}\label{wl}
  a^{p-1}\equiv 1  \Mod {p^2}.
\end{equation}
Otherwise, it is called a \textit{non-Wieferich prime for base $a$}. In $1909$, Arthur Wieferich \cite{wief} proved that if the first case of Fermat's last theorem is not true for a prime $p$, then $p$ is a Wieferich prime for base $2$. Today, $1093$ and $3511$ are the only known Wieferich primes for base $2$. It is still unknown whether there are infinitely many Wieferich primes that exist or not, for any given base $a$. But, Silverman proved that there are infinitely many non-Wieferich primes that exist for any base $a$, assuming the $abc$ conjecture.

\begin{thm}(Silverman \cite[Theorem 1]{sman})
For any fixed $a\in \mathbb{Q}^{*}$, where  $\mathbb{Q}^{*}=\mathbb{Q}\backslash\{0\}$ and $a\neq\pm1$. If the $abc$ conjecture is true, then
 \begin{equation*}
   \big|\{primes \, p\leq x: a^{p-1}
   \not\equiv 1 \Mod {p^2} \}\big| \gg_a \log x.
   \end{equation*}
\end{thm}

\noindent
In $2013$, Graves and Ram Murty extended Silverman's result to certain arithmetic progressions.
\begin{thm}(Graves and Ram Murty \cite[Theorem 3.3]{graves})
If $a,k$ and $n$ are positive integers and one assumes the abc conjecture, then
\begin{equation*}
   \big|\{primes \, p\leq x: p\equiv 1 \Mod k,\\a^{p-1}
   \not\equiv 1 \Mod {p^2}\}\big| \gg_{a,k} \frac{\log x}{\log\log x}.
\end{equation*}
\end{thm}

\noindent
Later, Chen and Ding \cite{chending} improved the lower bound to  $\displaystyle\frac{\log x (\log\log\log x)^M}{\log\log x}$, where $M$ is any fixed positive integer. Recently, Ding \cite{ding} further sharpened it to $\log x$.

\medskip
\noindent
We prove the similar lower bound for non-Wieferich primes in Lucas sequences under the assumption of the $abc$ conjecture for number fields. We first study the Lucas sequences and Lucas non-Wieferich primes.

\section{Lucas Sequences}

\begin{defn}\cite{mci}
	The \textit{Lucas sequences of first kind} $(U_n(P, \, Q))_{n\geq0}$ and the \textit{second kind} $(V_{n}(P, \, Q))_{n\geq0}$ are defined by the recurrence relations,
	\begin{align*}
		U_n(P,  Q) &= P U_{n-1}(P,  Q)-Q  U_{n-2}(P,  Q),\\
		V_{n}(P,  Q) &= P V_{n-1}(P,  Q)-Q\, V_{n-2}(P,  Q),
	\end{align*}
	for all $n\geq 2$ with initial conditions $U_0(P, Q)=0$ and $U_1(P, Q)=1$, $V_0(P, \, Q)=2$ and $V_1(P, \, Q)=P$. Here, $P$ and $Q$ are  non-zero fixed integers with $\gcd(P, Q)=1$. We always assume that the discriminant $\Delta:=P^2-4Q$ is positive.

\noindent
Alternatively, we have the formulae \cite{ribbook}

\begin{align}
	U_n(P,  Q) &= \frac{\alpha^n-\beta^n}{\alpha-\beta},\label{bin}\\
	V_{n}(P,  Q) &= \alpha^n+\beta^n,
\end{align}
\noindent
where $\alpha$ and $\beta$ are the zeros of the polynomial $x^2-Px+Q$, following the convention that $|\alpha| > |\beta|$. It is also called the \textit{Binet formulae}. 
	
Lucas sequences have some other directions in solving Diophantine equations of the symmetric form. In \cite{mwang}, Min Wang et al. solved Diophantine equations of the form $A_{n_1}\cdots A_{n_k}=B_{m_1} \cdots B_{m_r}C_{t_1} \cdots C_{t_s}$, where $(A_n), \, (B_m)$, and $(C_t)$ are Lucas sequences of first or second kind and also they found all the solutions of certain symmetric Diophantine equations.

\end{defn}

We note that the \textit{Fibonacci sequences} are particular case of Lucas sequences of first kind associated to the pair $(1,-1)$ and it is denoted by $(F_n)_{n\geq0}$ with initial conditions $F_0=0$ and $F_1=1$. The sequence of \textit{Lucas numbers} are example of Lucas sequence of  second kind associated to the pair $(1,-1)$ and it is denoted by $(L_n)_{n\geq0}$ with initial conditions $L_0=2$ and $L_1=1$ \cite{ribbook}. 

We observe that, 
$$
\displaystyle\lim_{n \to \infty}\frac{L_{n+1}}{L_n}=(1+\sqrt{5})/2=\lim_{n \to \infty}\frac{F_{n+1}}{F_n}.
$$
The number $(1+\sqrt{5})/2$ is called \textit{the golden ratio} \cite{koshy}. The Fibonacci sequences and the golden ratio are used in mathematical models that appear in the field of Chemistry. In particular, in the structure of elements, the periodic table of the elements, and their symmetrical crystalline structures \cite{wlo}, \cite{bat}.

Now, on returning to the general Lucas sequences of first kind $(U_n(P, \, Q))_{n\geq0}$. Throughout this paper, we simply write $U_n$ instead of $U_n(P,  Q)$, if $P$ and $Q$ are fixed. In 2007, McIntosh and Roettger \cite{mci} defined Lucas non-Wieferich primes and studied their properties. 

\begin{defn}(McIntosh and Roettger \cite{mci})
  An odd prime $p$ is called a \textit{Lucas--Wieferich prime associated to the pair $(P, Q)$} if
$$
   U_{p - \legendre{\Delta}{p}}\equiv 0 \Mod {p^2},
$$
where $\legendre{\Delta}{p}$ denotes the Legendre symbol. Otherwise, it is called a \textit{Lucas non-Wieferich prime associated to the pair $(P, Q)$}.
 \end{defn}
If $(P,Q) = (3,2)$, then $\Delta =1$ and $U_{p - \legendre{\Delta}{p}} =  2^{p-1}-1$. Every Wieferich prime is thus a Lucas--Wieferich prime associated to the pair $(3, 2)$ \cite{mci}. In $2001$, Ribenboim \cite{riben} proved that there are infinitely many Lucas non-Wieferich primes under the assumption of the $abc$ conjecture.

Recently, Rout \cite{rout2} proved some lower bound for the number of Lucas non-Wieferich primes $p$ such that $p\equiv1\Mod{k}$. However, we find some  gaps in his proofs. More clearly, he used the following lemma which is a classical result of cyclotomic polynomial $\Phi_m(x)$.

 \begin{lem}(Ram Murty \cite{ram})\label{rammurty}
	If $p|\Phi_m(a)$, then either $p|m$ or $p\equiv 1 \Mod m$.
\end{lem}

The above lemma is true for any rational integer $a$, but not true for any general algebraic integer. Further in Rout's proof (\cite[page 6, line 4]{rout2}), he mixed up the rational integer $a$ with the algebraic number $\alpha/\beta$. This was noted by Wang and Ding \cite{wang} in their paper.
We now give an example that illustrates how Lemma \ref{rammurty} fails when $a$ is replaced by $\alpha/\beta$.
\begin{exmp}
Let $\alpha=(7+\sqrt{37})/2$ and $\beta=(7-\sqrt{37})/2$. For $m=6$, the cyclotomic polynomial $\Phi_6(\alpha/\beta)=(860+140\sqrt{37})/9$.
For a prime $p=5$, $5\mid\Phi_6(\alpha/\beta)$. But $5\nmid6$ and $5\not\equiv1 \Mod{6}$.
\end{exmp}
In this paper, we will fill those gaps and prove that there are $\gg\log x$ Lucas non-Wieferich primes $p$ such that $p\equiv\pm1 \Mod{k}$ using Ding's \cite{ding} proof techniques. To the best of our knowledge, our main theorem is the first result which addresses the problem of Lucas non-Wieferich primes in arithmetic progressions. More precisely, we prove the following:

\section{Main Theorem}
\begin{thm}\label{theorem2}
  Let $k\geq2$ be a fixed integer and let $n>1$ be any integer. Assuming the $abc$ conjecture for number fields (defined below), then
$$
    \bigg|\bigg\{primes \,  p\leq x: p \equiv\pm 1 \Mod  k,\,
      U_{p-\legendre{\Delta}{p}}\not\equiv 0 \Mod { p^2}\bigg\} \bigg| \gg_{\alpha,  k}\log x.
$$
\end{thm}

\section{The $abc$ conjecture}
\subsection\normalfont{\textit{The $abc$ conjecture for integers \big( Oesterl\'{e},\,  Masser\big)}}

Given any real number $\varepsilon>0$, there is a constant $C_{\varepsilon}$ such that for every triple of coprime positive integers $a, \, b, \,c$ satisfying $a+b=c$, we have
\begin{equation*}
  c<C_{\varepsilon}(rad(abc))^{1+\varepsilon},
\end{equation*}
where $rad(abc)=\textstyle\prod\limits_{p|abc}p$.

We now recall the definition of Vinogradov symbol ( denoted as $\gg$).
\begin{defn}\cite{vojta}
	Let $f$ and $g$ are two non-negative functions. If $f<cg$ for some positive constant $c$, then we write $f\ll g$ or $g\gg f$. It is also called \textit{Vinogradov symbol}.	
\end{defn}
\subsection\normalfont{\textit{The generalized $abc$ conjecture for algebraic number fields (\cite{vojta}, \cite{gyo})}}

Let $K$ be an algebraic number field with ring of integers $\mathcal{O}_K$ and let $K^{*}=K\backslash\{0\}$. Let $V_K$ be the set of all primes on $K$, i.e., any $\upsilon\in V_K$ is an equivalence class of non-trivial norms on $K$ (finite or infinite). For $\upsilon\in V_K$, we define an absolute value $\|\cdot\|_\upsilon$ by

\medskip
\begin{displaymath}
   \|x\|_\upsilon = \begin{cases}
      |\psi(x)| & \text{if } \upsilon  \text{ is infinite,  corresponding }  \\
                & \text{to the real embedding } \psi:K\rightarrow\mathbb{C} ,\\
      |\psi(x)|^2 & \text{if } \upsilon \text{ is infinite, corresponding }  \\
                & \text{to the complex embedding } \psi:K\rightarrow\mathbb{C} ,\\
      N(\mathfrak{p})^{-ord_{\mathfrak{p}} (x)} & \text{if } v \text{ is finite and } \mathfrak{p} \text{ is the corresponding }  \\
                & \text{ prime ideal},
      \end{cases}
\end{displaymath}
for all $x \in K^*$. For $x\neq0, \, ord_{\mathfrak{p}} (x)$ denotes the exponent of $\mathfrak{p}$ in the prime ideal factorization of the principal fractional ideal $(x)$.

\noindent For any triple $(a, b,  c )\in K^*$, the \textit{height} of the triple is
$$
  H_{K}(a, b,  c):=\displaystyle\prod\limits_{\upsilon\in V_K}\max\big(\|a\|_{\upsilon},  \|b\|_{\upsilon}, \|c\|_{\upsilon}\big).
$$
The \textit{radical} of the triple $(a,  b,  c) \in K^*$ is
$$
  rad_K(a,  b,  c):=\displaystyle\prod\limits_{\mathfrak{p} \in I_{K}(a,  b, c)}N(\mathfrak{p})^{ord_\mathfrak{p}(p)},
$$
where $p$ is the rational prime lying below the prime ideal $\mathfrak{p}$ and $I_K(a, \, b, \,c)$ is the set of all prime ideals $\mathfrak{p}$ of $\mathcal{O}_K$ for which $\|a\|_\upsilon, \, \|b\|_\upsilon, \, \|c\|_\upsilon$ are not equal.
 
\medskip
\noindent
The $abc$ conjecture for an algebraic number field $K$ states that for any $\varepsilon>0$,
$$
   H_{K}(a, \,b, \, c)\ll_{\varepsilon,\, K}(rad_K(a, b,  c))^{1+\varepsilon},
$$
for all $a,  b,  c\in K^*$ satisfying $a+b+c=0$.
 \section{Preliminaries}

\noindent We need the following results for the proof of our main theorem. 
 \begin{lem}(Rout \cite[Corollary 3.3]{rout2}) \label{lem1}
Let $p$ be a prime and coprime to $2Q$. Suppose $U_n\equiv0 \Mod{p}$ and $U_n\not\equiv0 \Mod{p^2}$. Then $U_{p- \legendre{\Delta}{p}}\equiv0 \Mod{p}$ and $U_{p-\legendre{\Delta}{p}}\not\equiv0 \Mod{p^2}$.
 \end{lem}
That is, Lemma \ref{lem1} says that if a prime $p$ divides the square-free part of $U_n$, for some $n\in \mathbb{N}$, then $p$ is a Lucas non-Wieferich prime.

The \textit{rank of appearance (or apparition)} of a positive integer $k$ in the Lucas sequence $(U_n)_{n\geq0}$ is the least positive integer $m$ such that $k|U_{m}$. We denote the rank of apparition of $k$ by $\omega(k)$ if it exists \cite{lucas}. In 1930, Lehmer \cite{lehmer} proved that a prime $p$ divides $U_n$ if and only if $\omega(p)$ divides $n$. Thus, the prime $p$ is a Lucas non-Wieferich prime if and only if $\omega(p)$ divides $p-\legendre{\Delta}{p}$. 

\begin{lem}(Rout \cite[Lemma 3.4]{rout2})\label{lem2}
For sufficiently large $n\geq0$, we have

\begin{equation}\label{eqnlem2}
|\alpha|^{n/2} < |U_n| \leq 2 |\alpha|^n.
\end{equation}
\end{lem}


\noindent We now recall cyclotomic polynomial and some of its properties.
\begin{defn}\cite{ram}\label{cyclotomic}
Let $m\geq1$ be any integer. Then the $m^{th}$ \textit{cyclotomic polynomial} is
$$
  \Phi_m(X)=\prod\limits_{\substack{h=1 \\ \gcd(h,m)=1}}^{m}(X-\zeta_m^h),
$$
where $\zeta_m$ is the primitive $m^{th}$ root of unity.

It follows that,
\begin{equation}\label{pab}
X^m-1=\prod\limits_{\substack{d|m }}\Phi_d(X).
\end{equation}
\end{defn}

The following lemma characterizes the prime divisors of $\Phi_m(\alpha, \beta)$, \\ where $$
\Phi_m(\alpha,\beta)=\prod\limits_{\substack{h=1\\ \gcd(h,m)=1}}^{m} (\alpha-\zeta_m^h\beta).
$$
In the following lemma, let $P(r)$ denotes the greatest prime factor of $r$ with the convention that $P(0)=P(\pm1)=1$. 
  \begin{lem}(Stewart \cite[Lemma 2]{stewart})\label{rm}
  Let $(\alpha+\beta)^2$ and $\alpha\beta$ be coprime non-zero integers with $\alpha/\beta$ not a root of unity. If $m>4$ and $m\neq6,  12$ then $P(m/\gcd(3,m))$ divides $\Phi_m(\alpha,\beta)$ to at most the first power. All other prime factors of $\Phi_m(\alpha, \beta)$ are congruent to $\pm1 \Mod{m}$. Further if $m>e^{452}4^{67}$ then $\Phi_m(\alpha, \beta)$ has at least one prime factor congruent to $\pm1 \Mod{m}$.	
\end{lem}
We remark that, Yu. Bilu et al. \cite{bilu} reduced the above lower bound $e^{452}4^{67}$ to $30$. In the above lemma, Stewart \cite{stewart} considered the cyclotomic polynomial 
\begin{equation}\label{stewart cyclo}
	\alpha^m-\beta^m=\prod\limits_{\substack{d|m }}\Phi_d(\alpha, \beta).
\end{equation}
But we take the prime divisors $p$ of $\Phi_m(\alpha/\beta)$ such that $p\nmid\alpha\beta=Q$. So the prime divisors of $\Phi_m(\alpha/\beta)$ and the prime divisors of $\Phi_m(\alpha,\beta)$ are the same. Thus by using above Lemma \ref{rm}, the prime divisors of $\Phi_m(\alpha/\beta)$ are congruent to $\pm1 \Mod{m}$. 
\begin{lem}(Rout \cite[Lemma 2.10]{rout1})\label{rg}
 For any real number $\alpha/\beta$ with $|\alpha/\beta|>1$, there exists a constant $C>0$ such that
$$
    |\Phi_m(\alpha/\beta)| \geq C|\alpha/\beta|^{\phi(m)},
$$
where $\phi(m)$ is the Euler totient function.
\end{lem}

\section{Main Results}
Let $n>1$ be any integer and let $k \geq 2$ be any fixed integer. We always write $U_{nk}=X_{nk}Y_{nk}$, where $X_{nk},Y_{nk}$ are the square-free and powerful parts of $U_{nk}$ respectively. \\
Let us also take $X^\prime_{nk} = \gcd (X_{nk}, \Phi_{nk}(\alpha/\beta))$ and $  Y^\prime_{nk} = \gcd (Y_{nk}, \Phi_{nk}(\alpha/\beta))$.

\medskip
\noindent
We note that by using Binet formula \eqref{bin}, we write
\begin{align*}
	U_n &=\frac{\beta^n}{\alpha-\beta}\big((\alpha/\beta)^n-1\big)\\
	&=\frac{\beta^n}{\sqrt{\Delta}}\big((\alpha/\beta)^n-1\big).
\end{align*}
Thus
\begin{equation}\label{delta}
	(\alpha/\beta)^n-1|\sqrt{\Delta}U_n.
\end{equation}
We prove the following lemma, which is similar to the result in \cite{rout2}. For the purpose of completeness, we present the proof here.
\begin{lem}
  Assume that the $abc$ conjecture is true for the quadratic field $\mathbb{Q}(\sqrt{\Delta})$. Then for any $\varepsilon>0$, we have
$$
     |X^\prime_{nk}| |Q|^{\phi(nk)}\gg_{\varepsilon}|U_{\phi(k)}|^{2(\phi(n)-\varepsilon)}.
$$

\end{lem}
\begin{proof}
By the Binet formula \eqref{bin}, we have
\begin{equation}\label{bin2}
    \sqrt{\Delta}U_{nk}-\alpha^{nk}+\beta^{nk}=0.
\end{equation}
Now, by applying the $abc$ conjecture for the number field $K=\mathbb{Q}(\sqrt{\Delta})$ to the equation \eqref{bin2}, we have: 

For any $\varepsilon>0$, there exists a constant $C_{\varepsilon}$ such that
\begin{equation}\label{abc}
    H(\sqrt{\Delta}U_{nk}, -\alpha^{nk}, \beta^{nk}) \leq C_\varepsilon(rad (\sqrt{\Delta}U_{nk}, -\alpha^{nk},  \beta^{nk}))^{1+\varepsilon},
\end{equation}
where
\begin{align}
    rad (\sqrt{\Delta}U_{nk}, -\alpha^{nk},  \beta^{nk}) &= \prod_{\mathfrak{p}|Q\sqrt{\Delta}U_{nk}}N(\mathfrak{p})^{ord_\mathfrak{p} (p)}  \leq Q^2\Delta X^2_{nk}Y_{nk}\label{r}\\    \nonumber H(\sqrt{\Delta}U_{nk}, -\alpha^{nk},  \beta^{nk})&=\max\{|\sqrt{\Delta}U_{nk}|, |-\alpha^{nk}|, |\beta^{nk}|\}\cdot \max\{|-\sqrt{\Delta}U_{nk}|, |\alpha^{nk}|, |-\beta^{nk}|\}\\
\nonumber    &\geq|\sqrt{\Delta}U_{nk}||-\sqrt{\Delta}U_{nk}|=\Delta U^2_{nk}\\
    &=\Delta X^2_{nk}Y^2_{nk}.\label{h}
\end{align}
Substituting \eqref{r} and \eqref{h} in \eqref{abc}, we get
  \begin{equation}\label{ec}
    Y_{nk}\ll_{\varepsilon} U^{2\varepsilon}_{nk}.
  \end{equation}
By equation \eqref{pab} we have,
$$
\Phi_{nk}(\alpha/\beta)=\displaystyle\frac{(\alpha/\beta)^{nk}-1}{ \displaystyle\prod_{d|nk}\Phi_{d}(\alpha/\beta)}.
$$
Since $\Phi_{nk}(\alpha/\beta)|(\alpha/\beta)^{nk}-1$ and by using \eqref{delta} we write

$$
 \Phi_{nk}(\alpha/\beta)| U_{nk}\sqrt{\Delta}.
 $$
Since $U_{nk}=X_{nk}Y_{nk}$,
$$
\Phi_{nk}(\alpha/\beta)|X_{nk}Y_{nk}\sqrt{\Delta}.
$$
As $\Phi_{nk}(\alpha/\beta)\nmid\sqrt{\Delta}$, we have  $\Phi_{nk}(\alpha/\beta)|X_{nk}Y_{nk}$. Since  $\gcd(X_{nk},Y_{nk})=1$, we obtain $\Phi_{nk}(\alpha/\beta)$ divides $X_{nk}$ or $\Phi_{nk}(\alpha/\beta)$ divides $Y_{nk}$.

\medskip
 Suppose that $\Phi_{nk}(\alpha/ \beta)$ divides $X_{nk}$, we have $X^\prime_{nk}=\gcd(X_{nk}, \Phi_{nk}(\alpha/\beta))=\Phi_{nk}(\alpha/\beta)$ and $Y^\prime_{nk}=\gcd(Y_{nk},  \Phi_{nk}(\alpha/ \beta))=1$.
Similarly, if $\Phi_{nk}(\alpha/ \beta)$ divides $Y_{nk}$, we get $X^\prime_{nk} = 1$ and $Y^\prime_{nk} = \Phi_{nk}(\alpha/ \beta)$. Thus in either case, we obtain
\begin{equation}\label{ef}
  X^\prime_{nk}Y^\prime_{nk}=\Phi_{nk}(\alpha/ \beta).
\end{equation}
By Lemma \ref{rg} we write,
\begin{equation}\label{es}
   |X^\prime_{nk}Y^\prime_{nk}|=|\Phi_{nk}(\alpha/\beta)| \geq  C|\alpha/\beta|^{\phi(nk)}= C|\alpha^{2}/Q|^{\phi(nk)}.
\end{equation}
Hence from equations \eqref{ec}, \eqref{es} and \eqref{eqnlem2},
\begin{align*}
  |X^\prime_{nk}U^{2\varepsilon}_{nk}| &\gg_{\varepsilon}|X^{\prime}_{nk}Y_{nk}|\\
  &\geq |X^\prime_{nk}Y^\prime_{nk}|\\
  &\gg_{\varepsilon}\frac{1}{|Q|^{\phi(nk)}}|\alpha|^{2\phi(nk)}\\
  &\gg_{\varepsilon} \frac{1}{|Q|^{\phi(nk)}} | U_{\phi(k)}|^{2\phi(n)}. 
\end{align*}
Therefore,
\begin{align*}  
  |X^\prime_{nk}| &\gg_{\varepsilon}\frac{1}{|Q|^{\phi(nk)}}\bigg|\frac{U^{2\phi(n)}_{\phi(k)}}{U^{2\varepsilon}_{nk}}\bigg|\\
  &=\frac{1}{|Q|^{\phi(nk)}}\bigg|\frac{U_{\phi(k)}^{2\varepsilon}U_{\phi(k)}^{2(\phi(n)-\varepsilon)}}{U_{nk}^{2\varepsilon}}\bigg| \\
   &\gg_{\varepsilon} \frac{1}{|Q|^{\phi(nk)}} |U_{\phi(k)}|^{2(\phi(n)-\varepsilon)}.
\end{align*}
This completes the proof of the lemma.
\end{proof}

\medskip

%
\noindent The following lemma is inspired by the result in \cite[Lemma 2.4]{wang}.
 \begin{lem}\label{l8}
  If $m<n$, then $\gcd(X^\prime_{m}, X^\prime_{n})=1$ or a power of $\sqrt{\Delta}$.
\end{lem}
\begin{proof}	
	We suppose that $\gcd(X^{\prime}_m, X^{\prime}_n)>1$ is not a power of $\sqrt{\Delta}$. Let $\gamma(\neq\sqrt{\Delta})$ be a prime element of $\mathbb{Q}(\sqrt{\Delta})$ such that $\gamma|X^{\prime}_m$ and $\gamma|X^{\prime}_n$. By the definitions of $X^{\prime}_m$ and $X^{\prime}_n$, we write $\gamma|\Phi_m(\alpha/\beta)$ and $\gamma|\Phi_n(\alpha/\beta)$.
	
\medskip	
 Since $\Phi_m(\alpha/\beta)|(\alpha/\beta)^m-1$ and $\Phi_n(\alpha/\beta)|(\alpha/\beta)^n-1$, we get  $\gamma|(\alpha/\beta)^m-1$ and $\gamma|(\alpha/\beta)^n-1$. Thus
$\gamma|(\alpha/\beta)^{\gcd(m, n)}-1$.\\
 For $m<n, \, \gcd(m, n)<n$.\\
Hence
$$
 (\alpha/\beta)^n-1=\frac{(\alpha/\beta)^n-1}{(\alpha/\beta)^{\gcd(m, n)}-1}(\alpha/\beta)^{\gcd(m, n)}-1.
 $$
 As $\gcd(\Phi_n(\alpha/\beta), (\alpha/\beta)^{\gcd(m, n)}-1)=1$, we obtain
 $$ \Phi_n(\alpha/\beta)\bigg|\displaystyle\frac{(\alpha/\beta)^n-1}{(\alpha/\beta)^{\gcd(m, n)}-1}.
 $$
 It follows that $\gamma^2|(\alpha/\beta)^n-1$. Thus using \eqref{delta}, we get $\gamma^2|\sqrt{\Delta}U_n$. As $\gamma\neq\sqrt{\Delta}, \, \gamma^2|U_n$.
 Let $p$ be a rational prime such that $\gamma|p$. Since $X_n$ and $U_n$ are integers and $\gamma|X^{\prime}_n, X^{\prime}_n|X_n$. It follows that $p|X_n$ and $p^2|U_n$. This contradicts the definition of $X_n$. 
\end{proof}

\noindent We recall the following lemma from \cite{ding}. 
\begin{lem}(Ding \cite[Lemma 2.5]{ding})\label{lim}
For any given positive integers $k$ and $n$, we have
$$
\displaystyle\sum_{n\leq x}\frac{\phi(nk)}{nk}=c(k)x+O(\log x),
$$
where $c(k)=\prod\limits_{p}\big(1-\frac{\gcd(p, k)}{p^2}\big)>0$ and the implied constant depends on $k$.
\end{lem}

The following lemma is an analogous result of \cite[Lemma 2.6]{ding} for the Lucas sequences.
\begin{lem}\label{l11}
 Let $S=\Big\{n: |Q|^{\phi(nk)}|X^\prime_{nk}|>nk\Big\}$ and $S(x)=|S\cap[1, \, x]|$. Then, 
$$
S(x)\gg_{\alpha, k} x.
$$
\end{lem}
\begin{proof}
Let $T=\Big\{\displaystyle n:\, \phi(nk)>2c(k)nk/3\Big\}$ and $T(x)=|T\cap[1, \, x]|$. By using equations \eqref{eqnlem2} and \eqref{ec}, we have
 \begin{equation}\label{fs}
    |Y^\prime_{nk}|\leq |Y_{nk}|\ll_{\varepsilon}| U_{nk}|^{2\varepsilon}\leq2|\alpha^{nk}|^{2\varepsilon}.
  \end{equation}
On substituting \eqref{fs} in \eqref{es} and we write
 \begin{equation}\label{ne}
   |Q|^{\phi(nk)} |X^\prime_{nk}|\gg_{\varepsilon} |\alpha|^{2(\phi(nk)-\varepsilon nk)}.
  \end{equation}
By taking $\varepsilon=c(k)/3$ in equation \eqref{ne}, we obtain 
\begin{equation}\label{f0}
\displaystyle|Q|^{\phi(nk)}|X^\prime_{nk}|\gg_{\varepsilon} \displaystyle|\alpha|^{2(\phi(nk)-c(k)nk/3)}.
\end{equation}

For any $n\in T$, the above equation \eqref{f0} becomes,
$$
\textstyle|Q|^{\phi(nk)} |X^\prime_{nk}| \gg_{\varepsilon} \textstyle |\alpha|^{2(\phi(nk)-c(k)nk/3)}> |\alpha|^{2c(k)nk/3}\gg_{\alpha, k} |\alpha|^{2c(k)nk/3-\log nk/\log|\alpha|}.nk>nk.
$$
Hence there exists an integer $n_0$ depending only on $\alpha, \, k$ such that, if $n\geq n_0$ and $n\in T$, then $|Q|^{\phi(nk)}|X^\prime_{nk}|>nk$. 
Now we have,
 \begin{equation}\label{f1}
      S(x) = \sum\limits_{\substack{n\leq x \\ |Q|^{\phi(nk)}|X^\prime_{nk}|>nk}}1\geq \sum\limits_{\substack{n\leq x \\ n\geq n_0 \\ n\in T }}1
       = \sum\limits_{\substack{n\leq x \\ n\geq n_0 \\ \phi(nk)>2c(k)nk/3}}1.
    \end{equation}
    Since we have
    \begin{equation}\label{f2}
     \sum\limits_{\substack{n\leq x \\ \phi(nk)\leq2c(k)nk/3}} \displaystyle\frac{\phi(nk)}{nk} \leq  \sum\limits_{\substack{n\leq x \\ \phi(nk)\leq2c(k)nk/3}} \displaystyle\frac{2c(k)}{3}
     \leq\displaystyle\frac{2c(k)}{3}x.
    \end{equation}
 Hence by Lemma \ref{lim} and equation \eqref{f2} we shall write,
\begin{align*}
  S(x) &\geq \sum\limits_{\substack{n\leq x \\ n\geq n_0 \\ \phi(nk)>2c(k)nk/3}}1 \\
   &\gg \sum\limits_{\substack{n\leq x \\ \phi(nk)>2c(k)nk/3}}1 \\
   &\geq \sum\limits_{\substack{n\leq x \\ \phi(nk)>2c(k)nk/3}}\frac{\phi(nk)}{nk} \\
  &= \sum\limits_{\substack{n\leq x }}\frac{\phi(nk)}{nk}-\sum\limits_{\substack{n\leq x \\ \phi(nk)\leq2c(k)nk/3}}\frac{\phi(nk)}{nk}\\
   &\geq c(k)x+O(\log x)-\displaystyle\frac{2c(k)}{3}x\gg_{\alpha,k} x.
\end{align*}
This completes the proof of Lemma \ref{l11}.
\end{proof}
\section{Proof of Main Theorem}
For any $n\in S$, there exists a prime $p_n$ such that $p_{n}|X^{\prime}_{nk}$ and $p_n\nmid nk$. Since $p_n|X^{\prime}_{nk}$ and $X^{\prime}_{nk}|X_{nk}$, we observe that $p_n|U_{nk}$ and $p_n^2 \nmid U_{nk}$. 
Note that for all $p_n \nmid PQ \Delta$, except possibly finitely many primes. Hence by using Lemma \ref{lem1}, we obtain
\begin{equation*}
    U_{p_{n}-\legendre{\Delta}{p_n}}\not\equiv 0 \Mod { p_n^2}.
  \end{equation*}
 As $p_n|\Phi_{nk}(\alpha/\beta), \, p_n\nmid nk$ and by using Lemma \ref{rm}, we have $p_{n} \equiv \pm1 \Mod {nk}$. Hence for any $n\in S$, there is a prime $p_{n}$ satisfying
  \begin{align*}
     U_{p_{n}-\legendre{\Delta}{p_{n}}}&\not\equiv 0 \Mod {p_{n}^2}, \\
    p_{n} &\equiv\pm 1 \Mod {nk}.
  \end{align*}
  From Lemma \ref{l8}, we conclude that each $p_n \, (n\in S)$ are distinct prime.
Thus we explore that,
\begin{align*}
    \bigg|\bigg\{primes \, p\leq x\, : \,  p \equiv\pm 1 \Mod k, \,  
    U_{p-\legendre{\Delta}{p}}\not\equiv 0 \Mod {p^2}\bigg\}\bigg|
    &\geq \bigg|\bigg\{n\, :\, n\in S, \,|Q|^{\phi(nk)}|X^\prime_{nk}|\leq x\bigg\}\bigg|.
  \end{align*}
Since $|Q|=|\alpha\beta|$, we shall write $|Q|^{\phi(nk)}<|\alpha|^{2nk}$ and also we have $|X^\prime_{nk}|\leq |X_{nk}|\leq |U_{nk}|\leq2|\alpha|^{nk}$. Therefore we obtain $|Q|^{\phi(nk)}|X^\prime_{nk}|<2|\alpha|^{3nk}$. We now get
\begin{align*}
  \bigg|\bigg\{n\, :\, n\in S, \, |Q|^{\phi(nk)} |X^\prime_{nk}|\leq x\bigg\}\bigg| &\geq  \bigg|\bigg\{n\, :\, n\in S, \,2|\alpha|^{3nk}\leq x\bigg\}\bigg| \\
   &= \bigg|\bigg\{n\, :\, n\in S, \, n\leq \frac{\log x/2}{3k\log |\alpha|} \bigg\}\bigg| \\
   &= \textstyle S\big(\frac{\log x/2}{3k \log |\alpha|}\big).
\end{align*}
Hence by Lemma \ref{l11},
\begin{align*}
  \bigg|\bigg\{primes \, p\leq \, x\, : \,  p \equiv\pm 1 \Mod k,\, 
    U_{p-\legendre{\Delta}{p}}\not\equiv 0 \Mod { p^2}\bigg\}\bigg|
   &\geq \textstyle S\big(\frac{\log x/2}{3k \log |\alpha|}\big) \\
   &\gg_{\alpha, k}\log x/2 \\
   &\gg_{\alpha,  k} \log x.
   \end{align*}
This completes the proof.
\section{Applications of Lucas sequences in Cryptography}
  Many branches of number theory frequently deal with Lucas sequences. As an application in Cryptography, the various studies related to public-key encryption schemes based on the Lucas sequences have been highlighted. In \cite{smith}, Smith et al. introduced the \textit{Public-key cryptosystem (LUC)} that is based on the Lucas sequences. Further, they widely discussed the cryptography properties of Lucas sequences and the cryptographic strength of LUC.

The major advantage of Lucas-based cryptosystems is that they are not formulated in terms of exponentiation \cite{bosma}. Then, Jiang et al. \cite{jiang} proposed a variant of (probabilistic) public-key encryption scheme based on Lucas sequences and also they analyzed the efficiency of the proposed schemes. Recently, a cryptosystem that is analogous to an elliptic curve cryptosystem has been developed by using Lucas sequences of the second kind \cite{sarbini}. By using the results in this paper, we suggest further investigation in this direction.
\section{Conclusions}
In this paper, we considered Lucas sequences which are more general sequences of Fibonacci and Lucas numbers. These special sequences are related to other research areas such as Chemistry and Cryptography.  We proved that under the assumption of $abc$ conjecture for number fields, there are at least $O(\log x)$ as many Lucas non-Wieferich primes p such that $p\equiv\pm1 \Mod{k}$ for any fixed integer $k\geq2$. 

Our results will lead to further investigations in the field of Cryptography. Especially, in public-key cryptography or asymmetric cryptography. It will be more useful in engineering and technological applications.
 
\medskip
\noindent
\textbf{Acknowledgment.}
The author I. Mumtaj Fathima would like to express her gratitude to Maulana Azad National Fellowship for minority students, UGC. This research work is supported by MANF-2015-17-TAM-56982, University Grants Commission (UGC), Government of India.


\end{document}